\documentclass[12pt,a4paper]{amsart}
\usepackage{amsmath, amssymb, amsthm, yhmath}
\input cyracc.def

\usepackage{graphicx}
\usepackage{caption}
\usepackage{float}
\usepackage{mathtools}
\usepackage{xcolor}
\usepackage[colorlinks=false,
            linkbordercolor={0 0 1},
            citebordercolor={0 1 0},
            urlbordercolor={1 0 0}
]{hyperref}
\numberwithin{equation}{section}
\usepackage{enumerate}
\newtheorem{theorem}{Theorem}[section]
\newtheorem{lemma}[theorem]{Lemma}
\newtheorem{claim}{Claim}[section]
\newtheorem{corollary}[theorem]{Corollary}
\newtheorem{proposition}[theorem]{Proposition}
\theoremstyle{definition}

\newtheoremstyle{remarkstyle}
   {}
   {}
  {\normalfont}
  {}
  {\itshape}
  {.}
  { }
  {}
\theoremstyle{remarkstyle}
\newtheorem*{remark}{Remark}

\allowdisplaybreaks[4]

\newcommand{\supp}{\operatorname{supp}}

\begin{document}

\author[K. Matsuzaki]{Katsuhiko Matsuzaki \textsuperscript{$a$}}
\address{$a.$ Department of Mathematics, School of Education, Waseda University, Tokyo 169-8050, Japan}
\email{matsuzak@waseda.jp}

\author[F. Tao]{Fei Tao \textsuperscript{$b$}}
\address{$b.$ Beijing International Center for Mathematical Research, Peking University, Beijing 100871, P. R. China}
\email{ferrytau@pku.edu.cn}

\thanks{The first author is partially supported by Japan Society for the Promotion of Science (KAKENHI 23K25775 and 23K17656); the second author is partially supported by National Key R \& D Program of China (2025YFA1017500)}


\title[$\mathrm{BMO}$ Domains and Quasicircles Associated with Zygmund Functions]{$\mathrm{BMO}$ Domains and Quasicircles Associated with Zygmund Functions}

\begin{abstract}
We investigate graph domains and graph curves from the viewpoints of
$\mathrm{BMOA}$, Zygmund regularity, and asymptotic geometry.
First, we prove that every graph domain is a $\mathrm{BMO}$ domain, which provides a
proof of a statement of Astala and Zinsmeister.
We then use a Weierstrass-type function to construct a Bishop--Jones quasicircle that is asymptotically conformal but not asymptotically smooth. This example yields the strict inclusions
\[
 \mathrm{VMOA}(\mathbb H)
 \subsetneq \mathrm{BMOA}(\mathbb H)\cap B_0(\mathbb H),
 \qquad
 \mathrm{SS}(\mathbb R)
 \subsetneq \mathrm{SQS}(\mathbb R)\cap\mathrm{S}(\mathbb R).
\]
Finally, by a localized lacunary construction, we produce a compactly
supported locally absolutely continuous function $f$ with
$f'\in\mathrm{BMO}(\mathbb R)\setminus\mathrm{VMO}(\mathbb R)$ whose
graph is an asymptotically conformal chord-arc curve that is not
asymptotically smooth.
\end{abstract}

\maketitle

\section{Introduction}\label{intro}

Let $\mathbb H=\{z\in\mathbb{C}:\operatorname{Im}z>0\}$ and $\mathbb H^*=\mathbb{C}\setminus \overline{\mathbb{H}}$ denote the upper and lower half-planes, respectively.
Let $\Gamma\subset\mathbb{C}$ be a Jordan curve passing through $\infty$, and let $\Omega$ and $\Omega^*$ denote the two complementary components of $\Gamma$.
By the Riemann mapping theorem,
there exist conformal mappings
$F\colon\mathbb H\to\Omega$ and
$G\colon\mathbb H^*\to\Omega^*$
with
\[
F(\infty)=\lim_{z \to \infty}F(z)=\infty, \qquad G(\infty)=\lim_{z \to \infty}G(z)=\infty.
\]
By the Carath\'eodory theorem, their boundary extensions determine a sense-preserving homeomorphism
\[
 h=F^{-1}\circ G|_{\mathbb R}\colon\mathbb R\to\mathbb R,
\]
called the \emph{conformal welding} associated with $\Gamma$.

A classical theorem of Beurling and Ahlfors \cite{BA} implies that
$\Gamma$ is a quasicircle if and only if its conformal welding is
quasisymmetric. Recall that an increasing homeomorphism
$h\colon\mathbb R\to\mathbb R$ is \emph{quasisymmetric} if there exists $M\ge1$
such that
\[
 \frac1M
 \le
 \frac{h(x+t)-h(x)}{h(x)-h(x-t)}
 \le M
\]
for all $x\in\mathbb R$ and $t>0$. The class of such mappings is
denoted by $\mathrm{QS}(\mathbb R)$.

Several finer subclasses of $\mathrm{QS}(\mathbb R)$ reflect increasingly
strong regularity of the corresponding conformal maps and boundary
curves. A quasisymmetric homeomorphism is called \emph{symmetric} if
\[
 \frac{h(x+t)-h(x)}{h(x)-h(x-t)}
 \longrightarrow 1
 \qquad (t\to0)
\]
uniformly in $x\in\mathbb R$, and the corresponding class is denoted by
$\mathrm{S}(\mathbb R)$. Symmetric homeomorphisms were studied by
Gardiner and Sullivan \cite{GS} in connection with the little universal
Teichm\"uller space. More recently, Brania and Yang \cite{BY}
formulated the related notion of asymptotically symmetric embeddings
between metric spaces, in which asymptotic symmetry is described
through distortion at arbitrarily small scales.

A homeomorphism $h\in\mathrm{QS}(\mathbb R)$ is \emph{strongly quasisymmetric}
if it is locally absolutely continuous and $h'$ is a Muckenhoupt
$A_\infty$-weight. In particular,
$\log h'\in\mathrm{BMO}(\mathbb R)$.
The corresponding class is denoted by $\mathrm{SQS}(\mathbb R)$.
If, moreover,
$\log h'\in\mathrm{VMO}(\mathbb R)$,
then $h$ is called \emph{strongly symmetric}, and this class is denoted by
$\mathrm{SS}(\mathbb R)$. Thus
\[
 \mathrm{SS}(\mathbb R)
 \subset \mathrm{SQS}(\mathbb R)\cap\mathrm{S}(\mathbb R).
\]
These classes can be viewed as distinguished subspaces of the universal
Teichm\"uller space, which may be identified with
the quotient of $\mathrm{QS}(\mathbb{R})$ by the affine group
\[
\mathrm{Aff}(\mathbb{R})
=\{x\mapsto ax+b:\,a>0,\ b\in\mathbb{R}\}
\]
(see, for example, \cite{Le}). They are closely related to the classical analytic
function spaces $\mathrm{BMOA}$, $\mathrm{VMOA}$, the Bloch space, and the little Bloch space (for precise definitions, see Section~\ref{sec:functions_embeddings_and_curves}).

Among Jordan curves passing through $\infty$, graphs of continuous real-valued
functions form a particularly concrete class in which these analytic
and geometric conditions can be compared directly. In \cite{MT2}
we studied this comparison by means of graph welding, and several results of \cite{MT2}
are used below. The purpose of the present paper is to investigate the relation
between a function $f\colon\mathbb R\to\mathbb R$, the geometry of its graph,
and the boundary behavior of the corresponding conformal mappings from
two complementary directions. The first concerns the $\mathrm{BMOA}$ regularity
of conformal maps onto arbitrary graph domains. The second concerns
the construction of graph curves whose asymptotic geometry separates
several natural vanishing regularity conditions.

Let $F\colon\mathbb H\to\Omega$ be conformal, with $F(\infty)=\infty$.
We say that $\Omega$ is a \emph{$\mathrm{BMO}$ domain} if
$\log F'\in\mathrm{BMOA}(\mathbb H)$.
For bounded Jordan domains the definition is analogous, with $\mathbb H$ replaced by the unit disk $\mathbb{D}$.
Zinsmeister \cite{Zin} proved that every Jordan domain with
Ahlfors regular boundary is a $\mathrm{BMO}$ domain.
The converse is false
(see \cite[Exercise 17, p.~267]{GM}). Bishop and Jones
\cite[Theorem 4]{BJ} subsequently gave geometric characterizations of
$\mathrm{BMO}$ domains in terms of chord-arc subdomains and harmonic measures.

Graph domains need not have Ahlfors regular boundaries, and the
defining function is not assumed to satisfy any quantitative regularity.
Nevertheless, their geometry alone forces the $\mathrm{BMOA}$ property. Our first
result is that every graph domain is a $\mathrm{BMO}$ domain.
This fact was stated without proof by Astala and Zinsmeister
\cite[p.~617]{AZ}. Since we were unable to locate a proof in the
literature, we give one in Section~\ref{bmo_domain}. The argument is based on a
variant of the McMillan twist theorem. More precisely, for a conformal
mapping $F\colon\mathbb H\to\Omega$ onto a graph domain we prove that every
continuous branch of $\arg F'$ is bounded. The imaginary part of $\log F'$ is
then a bounded harmonic function, and the Carleson measure characterization of
$\mathrm{BMOA}$ yields $\log F'\in\mathrm{BMOA}(\mathbb H)$.
Thus no smoothness, rectifiability, or quantitative geometric
assumption on the graph is required.

We next turn to the small-scale geometry of graphs of Zygmund
functions. If $f$ belongs to the Zygmund class, the mapping
\[
F\colon x\mapsto x+if(x)
\]
is a quasisymmetric embedding of $\mathbb R$ into $\mathbb C$ \cite[Theorem 1.4 and Remark 3.1]{KO}, so that its graph
$\Gamma=F(\mathbb R)$ is a quasicircle. The little Zygmund condition
leads to asymptotic symmetry of $F$ and hence to asymptotic conformality of
the graph. Lacunary trigonometric series are a natural source
of explicit examples, since their second differences can be small while
their first-order behavior remains highly oscillatory.

In Section~\ref{Weierstrass_Zygmund} we consider the Weierstrass-type function
\[
 f(x)
 =
 \sum_{n=1}^{\infty}
 \frac{1}{2^n\sqrt n}\cos(2^n x).
\]
We prove that $f$ belongs to the little Zygmund class, so that its graph
is an asymptotically conformal quasicircle. At the same time, the
classical differentiability properties of lacunary series imply that the
graph is not locally rectifiable and hence cannot be asymptotically
smooth. Since every graph domain is a $\mathrm{BMO}$ domain by our first result,
this produces an explicit asymptotically conformal
Bishop--Jones quasicircle that is not asymptotically smooth.

Through the boundary characterizations of the little Bloch and $\mathrm{VMOA}$
classes on the half-plane obtained in \cite{MT1}, this example
also gives
\[
 \mathrm{VMOA}(\mathbb H)
 \subsetneq
 \mathrm{BMOA}(\mathbb H)\cap B_0(\mathbb H).
\]
In the corresponding Teichm\"uller-theoretic formulation, it gives
\[
 \mathrm{SS}(\mathbb R)
 \subsetneq
 \mathrm{SQS}(\mathbb R)\cap\mathrm{S}(\mathbb R).
\]
Thus asymptotic smoothness can fail even when both
$\mathrm{BMO}$-type regularity and asymptotic conformality are present. This is
a graph analogue, in the nonrectifiable setting, of the phenomenon studied in our recent work
\cite{MT0}, where an asymptotically conformal chord-arc curve that
is not asymptotically smooth was obtained by an iterative geometric
construction. The example also gives a negative answer to the question raised in \cite{MT0}
of whether every asymptotically conformal Bishop--Jones quasicircle is necessarily
a chord-arc curve, since the graph of $f$ is not even locally rectifiable.

A further aim of the present paper is to realize this phenomenon by a
much more regular graph function, which is carried out in Section~\ref{sec:construction_of_the_lacunary_series}.
We construct a compactly supported, locally absolutely continuous function
$f\colon\mathbb R\to\mathbb R$ in the little Zygmund class such that
\[
 f'\in\mathrm{BMO}(\mathbb R)\setminus\mathrm{VMO}(\mathbb R).
\]
The construction is based on suitably
localized lacunary trigonometric sums. Roughly speaking, the
coefficients of the individual blocks tend to zero, which makes the
Zygmund oscillation vanish at small scales, while the number of frequencies in each block
grows fast enough to keep the mean oscillation of
the derivative bounded away from zero on a sequence of shrinking
intervals. Since $f'\in\mathrm{BMO}(\mathbb R)$, the graph of $f$ is a
chord-arc curve, and its little Zygmund regularity makes it
asymptotically conformal. Nevertheless, the oscillation of $f'$ on
the distinguished shrinking intervals keeps the ratio of arc length to chord
length bounded away from $1$, so the graph is not asymptotically smooth.
In particular, the obstruction to asymptotic smoothness in this
example is not caused by nonrectifiability, since it persists within the
class of chord-arc graphs.

The paper is organized as follows. Section~\ref{sec:functions_embeddings_and_curves} collects the necessary
background on $\mathrm{BMO}$, $\mathrm{VMO}$, Bloch and $\mathrm{BMOA}$ spaces, Zygmund classes,
quasisymmetric and asymptotically symmetric embeddings, and the
corresponding classes of curves. It also contains a proof that the little Zygmund
class coincides with the strongly little Zygmund class. In Section~\ref{bmo_domain} we prove that every
graph domain is a $\mathrm{BMO}$ domain, together with the conformal distortion
estimates needed for the proof. Section~\ref{Weierstrass_Zygmund} studies graphs of Zygmund
functions and the Weierstrass-type example, which leads to the strict inclusions
above. Section~\ref{sec:construction_of_the_lacunary_series}
develops the localized lacunary construction.

\section{Functions, embeddings, and curves}\label{sec:functions_embeddings_and_curves}

\subsection{Function spaces}
A locally integrable function $u \in L^1_{\mathrm{loc}}(\mathbb{R})$ is said to be of \emph{bounded mean oscillation} if
\[
\|u\|_{\mathrm{BMO}(\mathbb{R})}\coloneqq\sup_{I \subset \mathbb{R}}\frac{1}{|I|}\int_I |u(x) - u_I|\,dx< \infty,
\]
where the supremum is taken over all bounded intervals $I \subset \mathbb{R}$,
and $u_I$ denotes the integral mean of $u$ over $I$. The collection of all such functions is denoted by $\mathrm{BMO}(\mathbb{R})$, which becomes a Banach space after identifying functions that differ by a constant. A function $u \in \mathrm{BMO}(\mathbb{R})$ is said to be of \emph{vanishing mean oscillation} if
\[
\lim_{\delta \to 0}\,\sup_{|I|\le\delta}\frac{1}{|I|}\int_I |u(x) - u_I|\,dx= 0.
\]
The set of all such functions is denoted by $\mathrm{VMO}(\mathbb{R})$.

A holomorphic function $\Phi$ on $\mathbb H$ is called a \emph{Bloch function} if
\[
\sup_{z \in \mathbb H}(\operatorname{Im} z)|\Phi'(z)| <\infty.
\]
The Banach space of all Bloch functions (modulo additive constants) is denoted by $B(\mathbb H)$. If $\Phi \in B(\mathbb H)$ satisfies
\[
\lim_{y \to 0}\,\sup_{x\in\mathbb R}\, y\,|\Phi'(x+iy)| =0,
\]
then $\Phi$ is called a \emph{little Bloch function}. The collection of all such functions forms a closed subspace of $B(\mathbb H)$, denoted by $B_0(\mathbb H)$.

A holomorphic function $\Phi$ on $\mathbb{H}$ belongs to $\mathrm{BMOA}(\mathbb{H})$ if the measure $|\Phi'(z)|^{2}\, y\, dx\, dy$ ($z=x+iy$) is a Carleson measure on $\mathbb{H}$, and it belongs to $\mathrm{VMOA}(\mathbb{H})$ if this measure is a vanishing Carleson measure. Equivalently, $\Phi\in\mathrm{BMOA}(\mathbb{H})$ if and only if $\Phi$ has non-tangential boundary values $\Phi^*$ almost everywhere on $\mathbb R$, $\Phi^*\in\mathrm{BMO}(\mathbb R)$, and $\Phi$ is the Poisson integral of $\Phi^*$. In this case $\Phi\in\mathrm{VMOA}(\mathbb{H})$ if and only if $\Phi^*\in\mathrm{VMO}(\mathbb R)$ (see \cite[Chapter VI]{Ga}).

\subsection{Zygmund classes}
For a function $u\colon\mathbb R\to\mathbb R$ we write the second difference as
\[
  \Delta_h^2u(x)\coloneqq u(x+h)+u(x-h)-2u(x).
\]
A continuous function $f\colon \mathbb{R}\to\mathbb{R}$ is said to belong to the \emph{Zygmund class}
$\Lambda_*(\mathbb{R})$ if
\[
\|f\|_{\Lambda_*}
\coloneqq
\sup_{x\in\mathbb{R},\,h\neq 0}
\frac{|\Delta_h^2f(x)|}{|h|}
<\infty .
\]
The quantity $\|f\|_{\Lambda_*}$ is called the \emph{Zygmund seminorm} of $f$. A function $f \in \Lambda_*(\mathbb{R})$
is said to belong to the \emph{little Zygmund class} $\lambda_*(\mathbb{R})$
if
\[
\lim_{h \to 0}\sup_{x \in \mathbb{R}}\frac{|\Delta_h^2f(x)|}{|h|}= 0.
\]

For $\lambda,\mu>0$ with $\lambda+\mu=2$ and $h\ne0$, set
\[
E_f^{\lambda,\mu}(x,h)\coloneqq\mu\bigl(f(x+\lambda h)-f(x)\bigr)-\lambda\bigl(f(x)-f(x-\mu h)\bigr).
\]
For $\lambda=\mu=1$ this is $\Delta_h^2 f(x)$. A function $f\in\Lambda_*(\mathbb R)$ belongs to the
\emph{strongly little Zygmund class} $\lambda^{S}_*(\mathbb{R})$ if
\[
\lim_{h \to 0}\,\sup_{x \in \mathbb{R}}\,\sup_{\substack{\lambda,\mu>0\\ \lambda+\mu=2}}\frac{|E_f^{\lambda,\mu}(x,h)|}{|h|}=0 .
\]
It is shown in \cite[Lemma 5.1]{MT2} that if $f\in \lambda^{S}_*(\mathbb{R})$, then $F(x)=x+if(x)$ is an asymptotically symmetric homeomorphism of $\mathbb R$ onto
the graph $\Gamma=F(\mathbb R)$ of $f$. The following estimate for the deviation of a function from its chords
is used repeatedly below. In particular, it shows that the two little Zygmund conditions coincide.

For a continuous function $g\colon\mathbb R\to\mathbb R$ and $r>0$, put
\[
\omega_g(r)
\coloneqq
\sup_{\substack{x\in\mathbb{R}\\0<|t|\le r}}
\frac{|\Delta_t^2 g(x)|}{|t|}\in[0,\infty].
\]
Then $\omega_g$ is nondecreasing, $\omega_g(r)\le\|g\|_{\Lambda_*}$ if $g\in \Lambda_*(\mathbb R)$, and $g\in\lambda_*(\mathbb R)$ if and only if $g\in\Lambda_*(\mathbb R)$ and $\omega_g(r)\to0$ as $r\to0$.

\begin{proposition}\label{little-strongly}
Let $g\colon\mathbb R\to\mathbb R$ be continuous, let $p<q$, and let $L$ be the affine function with
$L(p)=g(p)$ and $L(q)=g(q)$. Then
\[
|g-L|\le \frac{q-p}{2}\,\omega_g\Bigl(\frac{q-p}{2}\Bigr)\quad\text{on }[p,q].
\]
Consequently,
$|E_g^{\lambda,\mu}(x,h)|\le 2|h|\,\omega_g(|h|)$
for all $x\in\mathbb R$, $h\ne0$, and $\lambda,\mu>0$ with $\lambda+\mu=2$, and
$\lambda_*(\mathbb R)=\lambda_*^S(\mathbb R)$.
\end{proposition}

\begin{proof}
Put $h=(q-p)/2$. We may assume $\omega_g(h)<\infty$.
For each integer $k\ge0$, let
$x_{k,j}=p+2^{-k}j(q-p)$, $j=0,1,\ldots,2^k$,
and let $L_k$ be the continuous piecewise affine function with
$L_k(x_{k,j})=g(x_{k,j})$ for all $j$, so that $L_0=L$. On each interval
$[x_{k,j},x_{k,j+1}]$ the functions $L_k$ and $L_{k+1}$ agree at the endpoints, and
the only new interpolation point of $L_{k+1}$ is the midpoint $m_{k,j}$. Since
$x_{k,j}=m_{k,j}-2^{-k}h$ and $x_{k,j+1}=m_{k,j}+2^{-k}h$,
\[
|L_{k+1}(m_{k,j})-L_k(m_{k,j})|
=
\frac12\,\bigl|\Delta^2_{2^{-k}h}\,g(m_{k,j})\bigr|
\le
\frac{h}{2^{k+1}}\,\omega_g\!\left(\frac{h}{2^k}\right)
\le
\frac{h}{2^{k+1}}\,\omega_g(h).
\]
The function $L_{k+1}-L_k$ is affine on each half of $[x_{k,j},x_{k,j+1}]$
and vanishes at the endpoints of this interval, so its maximum
absolute value there is attained at $m_{k,j}$. Hence
\[
\|L_{k+1}-L_k\|_{L^\infty([p,q])}
\le
\frac{h}{2^{k+1}}\,\omega_g(h).
\]
Since $g$ is uniformly continuous on $[p,q]$ and the mesh of the partitions tends to zero,
$L_k\to g$ uniformly on $[p,q]$. Summing over $k$ gives
\[
\|g-L\|_{L^\infty([p,q])}
\le
\sum_{k=0}^{\infty}
\|L_{k+1}-L_k\|_{L^\infty([p,q])}
\le
h\,\omega_g(h).
\]

For the second assertion, since
$E_g^{\lambda,\mu}(x,-h)=E_g^{\mu,\lambda}(x,h)$,
it is enough to consider $h>0$. Fix $x\in\mathbb{R}$,
and put $p=x-\mu h$ and $q=x+\lambda h$, so that $q-p=2h$.
Since $x-p=\mu h$ and $q-x=\lambda h$, the affine function $L$ interpolating $g$ at $p$ and $q$ satisfies
\[
L(x)
=
\frac{q-x}{q-p}\,g(p)
+
\frac{x-p}{q-p}\,g(q)
=
\frac{\lambda}{2}\,g(p)+\frac{\mu}{2}\,g(q),
\]
and consequently
\[
E_g^{\lambda,\mu}(x,h)
=
\mu g(q)+\lambda g(p)-2g(x)
=
2\bigl(L(x)-g(x)\bigr).
\]
The first assertion now gives $|E_g^{\lambda,\mu}(x,h)|\le 2h\,\omega_g(h)$.
Finally, $E_g^{1,1}(x,h)=\Delta_h^2g(x)$ gives the inclusion $\lambda_*^S(\mathbb R)\subset\lambda_*(\mathbb R)$, while the
second assertion gives the reverse inclusion.
\end{proof}

\subsection{Embeddings and curves}
A \emph{quasisymmetric embedding} is an embedding $g\colon X \to Y$ between metric spaces for which there exists an increasing homeomorphism $\eta\colon [0,\infty) \to [0,\infty)$ such that
\begin{equation}\label{eq:qs}
\frac{|x-a|}{|x-b|} \leq t \implies \frac{|g(x)-g(a)|}{|g(x)-g(b)|} \leq \eta(t)
\end{equation}
for all distinct points $x,a,b \in X$ and all $t>0$.
Here $|x-y|$ stands for the distance between $x$ and $y$.
When $g$ is surjective, we call it a \emph{quasisymmetric homeomorphism}.
If $g$ satisfies \eqref{eq:qs} only for $t=1$, then $g$ is called \emph{weakly quasisymmetric}.
(Weak) quasisymmetry on $\mathbb R$ was introduced by Beurling and Ahlfors \cite{BA}
in their study of the boundary correspondence of quasiconformal self-maps
of $\mathbb H$, and
Tukia and V\"ais\"al\"a \cite{TV}
formulated the general definition in the setting of metric spaces.

When $X$ is $\mathbb R$ or $\mathbb S$ and $Y=\mathbb C$, weak quasisymmetry implies quasisymmetry
\cite{TV}.
If $g\colon \mathbb{R}\to\mathbb{C}$ is a quasisymmetric embedding, then $g(\mathbb{R})$ is a quasicircle, that is,
the image of the real axis under a quasiconformal self-map of the plane.

A quasisymmetric embedding (or homeomorphism)
$g\colon X \to Y$ between metric spaces is called \emph{asymptotically symmetric} if for all $\varepsilon > 0$ and $t > 0$ there exists $\delta > 0$ such that for all distinct points $x, a, b \in X$ contained in a ball of radius $\delta$,
\begin{equation}\label{eq:AS}
\frac{|x - a|}{|x - b|} \leq t \implies \frac{|g(x) - g(a)|}{|g(x) - g(b)|} \leq (1 + \varepsilon) t .
\end{equation}
This notion was introduced by
Brania and Yang \cite{BY}. In their paper an asymptotically symmetric embedding is not assumed to be quasisymmetric, whereas throughout this paper we require it to be quasisymmetric. If \eqref{eq:AS} is required only for $t=1$, the embedding is called \emph{weakly asymptotically symmetric}.

Let $\Gamma$ be a Jordan curve passing through $\infty$. For distinct points $a, b \in \Gamma$, the bounded subarc of $\Gamma$ between $a$ and $b$ is denoted by $\wideparen{ab}$, and its length by $\ell(\wideparen{ab})$. A quasicircle $\Gamma$ is said to be \emph{asymptotically conformal} if
\[
\lim_{r\to0}\ \sup_{0<|a-b|\le r}\ \max_{w \in \wideparen{ab}} \frac{|a-w|+|w-b|}{|a-b|}=1.
\]
A locally rectifiable curve $\Gamma$ is called \emph{chord-arc} if there exists $C \geq 1$ such that
$\ell(\wideparen{ab})\le C|a-b|$ for all $a,b \in \Gamma$, and \emph{asymptotically smooth} if
\[
\lim_{r\to0}\ \sup_{0<|a-b|\le r}\frac{\ell(\wideparen{ab})}{|a-b|}=1.
\]

\subsection{Graph domains}
Let $f\colon\mathbb R\to\mathbb R$ be continuous and let $\Gamma=\{x+if(x):x\in\mathbb R\}$ be its graph.
Then $\Gamma\cup\{\infty\}$ is a Jordan curve in the Riemann sphere, and the two complementary components
\[
\Omega^+=\{x+iy:y>f(x)\},\qquad \Omega^-=\{x+iy:y<f(x)\}
\]
are called \emph{graph domains}. If $F\colon\mathbb H\to\Omega^{\pm}$ is conformal with $F(\infty)=\infty$, then by the Carath\'eodory theorem
$F$ extends to a homeomorphism of $\overline{\mathbb H}$ onto $\overline{\Omega^{\pm}}$ that maps $\mathbb R$ onto $\Gamma$.
We use the same letter $F$ for this extension.

\section{$\mathrm{BMO}$ domains}\label{bmo_domain}

In this section we prove that every graph domain is a $\mathrm{BMO}$ domain.
We show that for a conformal map $F\colon\mathbb H \to \Omega$ onto a graph domain with $F(\infty)=\infty$,
every continuous branch of $\arg F'$ is bounded on $\mathbb H$. This strategy is suggested as an exercise in \cite[p.~172]{Pom}.
The proof follows that of the McMillan twist theorem (see \cite[Theorem 6.18]{Pom}). Since $\partial \Omega$
is a graph, the conformal map $F$ cannot twist near the boundary.

We use the hyperbolic metrics of curvature $-1$ on $\mathbb H$ and on the unit disk $\mathbb D=\{\xi\in\mathbb C\colon |\xi|<1\}$, whose length elements are $|d\xi|/\operatorname{Im}\xi$ and $2|d\xi|/(1-|\xi|^2)$, respectively. The corresponding distances are
\[
d_{\mathbb H}(z,w)=\operatorname{arccosh}
\!\left(
1+
\frac{|z-w|^2}{2\,\operatorname{Im}z\,\operatorname{Im}w}
\right),
\qquad
d_{\mathbb D}(\xi,\eta)=\log\frac{1+|\xi-\eta|/|1-\bar{\xi}\eta|}{1-|\xi-\eta|/|1-\bar{\xi}\eta|}.
\]
In particular, $d_{\mathbb D}(0,\xi)=\log\frac{1+|\xi|}{1-|\xi|}$ and $\tanh d_{\mathbb D}(0,\xi)=\frac{2|\xi|}{1+|\xi|^2}$.

Throughout this section we use the following maps. For $a\in\mathbb H$, let
\[
T_a(w)=\frac{w-\operatorname{Re}a}{\operatorname{Im}a},
\]
which is an affine automorphism of $\mathbb H$ with $T_a(a)=i$, and let
\[
\kappa(\xi)=i\,\frac{1+\xi}{1-\xi},\qquad \kappa^{-1}(w)=\frac{w-i}{w+i},
\]
so that the Cayley map $\kappa$ sends $\mathbb D$ conformally onto $\mathbb H$ with $\kappa(0)=i$ and $\kappa'(0)=2i$.
Both $T_a$ and $\kappa$ are isometries for the hyperbolic distances above.
For a conformal map $F$ on $\mathbb H$ we write $d_F(z)=\operatorname{dist}(F(z),\partial F(\mathbb H))$, and for $a\in\mathbb H$ we put
\[
\Theta_a=F\circ T_a^{-1}\circ\kappa ,
\]
which is univalent on $\mathbb D$ with $\Theta_a(0)=F(a)$ and $\Theta_a'(0)=2i(\operatorname{Im}a)F'(a)$.

We prepare three propositions, which are distortion theorems of different types for conformal mappings.
The first one is a consequence of the Koebe distortion theorem and appears as an exercise in \cite[p.~13]{Pom}.
We include a proof for completeness.

\begin{proposition}\label{fz-f0/f':label}
If $F$ maps the unit disk $\mathbb{D}$ conformally into $\mathbb{C}$, then
\[
\left| \log \frac{F(z) - F(0)}{z F'(0)} \right| \leq 4 d_{\mathbb{D}}(0,z)
\]
for all $z \in \mathbb{D}$, where the logarithm is the holomorphic branch on $\mathbb D$ that vanishes at $z=0$, the quotient being interpreted as $1$ at $z=0$.
\end{proposition}

\begin{proof}
Fix $z\in\mathbb{D}$ and define
\[
\widetilde F(w)=\frac{F\!\left(\frac{z+w}{1+\bar{z}w}\right)-F(z)}{(1-|z|^{2})F'(z)},\quad w\in\mathbb{D}.
\]
Then $\widetilde F$ is univalent on $\mathbb{D}$ with $\widetilde F(0)=0$ and $\widetilde F'(0)=1$, and the growth theorem \cite[Theorem 1.3]{Pom} gives
\[
\frac{|w|}{(1+|w|)^{2}}\le|\widetilde F(w)|\le\frac{|w|}{(1-|w|)^{2}},\quad w\in\mathbb{D}.
\]
Taking $w=-z$ and multiplying by $(1-|z|^{2})/|z|$, we find that
\begin{equation}
\label{eq:key}
\frac{1-|z|}{1+|z|}\le\left|\frac{F(z)-F(0)}{zF'(z)}\right|\le\frac{1+|z|}{1-|z|}
\end{equation}
for every $z\in\mathbb D\setminus\{0\}$.

Let $Q(z)=\log\bigl((F(z)-F(0))/(zF'(0))\bigr)$ with $Q(0)=0$, which is analytic on $\mathbb{D}$, and let
$\Phi(z)=zF'(z)/(F(z)-F(0))$ with $\Phi(0)=1$. Then $\Phi$ is analytic on $\mathbb D$ and
\[
Q'(z)
=\frac{F'(z)}{F(z)-F(0)}-\frac{1}{z}
=
\frac{\Phi(z)-1}{z}.
\]
Fix $z\in\mathbb D\setminus\{0\}$, put $r=|z|$ and $R=(1+r)/2$. For $|\zeta|<R$, the lower bound in \eqref{eq:key} gives
$|\Phi(\zeta)|\le(1+R)/(1-R)$, hence $|\Phi(\zeta)-1|\le 2/(1-R)$.
The function
$\Psi(\zeta)=\frac{1-R}{2}\bigl(\Phi(\zeta)-1\bigr)$
is therefore analytic on $\{|\zeta|<R\}$ with $|\Psi|\le1$ and $\Psi(0)=0$, and the Schwarz lemma on this disk yields
$|\Psi(z)|\leq |z|/R$. Consequently,
\[
|Q'(z)|=\frac{|\Phi(z)-1|}{|z|}
\leq
\frac{2}{R(1-R)}
=
\frac{8}{1-|z|^2}.
\]
For $z=re^{i\theta}$, integrating along the radius $t\mapsto te^{i\theta}$, $0\leq t\leq r$, gives
\[
|Q(z)|
\leq
8\int_0^r \frac{dt}{1-t^2}
=4\log\frac{1+r}{1-r}=4d_{\mathbb D}(0,z),
\]
which completes the proof.
\end{proof}

The next proposition is the half-plane version of the length distortion theorem \cite[Corollary~4.18]{Pom}.
It follows from that result by the conformal invariance of harmonic measure, since the map $\xi\mapsto T_{z_0}^{-1}(\kappa(\xi))$ carries radii of $\mathbb D$ to hyperbolic geodesics in $\mathbb H$ issuing from $z_0$.

\begin{proposition}\label{uhp-f(S):label}
Let $F\colon \mathbb H\to\mathbb C$ be conformal, let $z_0\in\mathbb H$, and let $I\subset\mathbb R$ be a nondegenerate bounded interval.
Then there exists a hyperbolic geodesic ray $S\subset\mathbb H$
joining $z_0$ to some point $\zeta\in I$ such that
\[
\ell(F(S))\le C_1\, d_F(z_0)\exp\!\left(\frac{C_2}{\omega_{\mathbb H}(z_0,I)}\right),
\]
where $\ell$ denotes Euclidean length, $\omega_{\mathbb H}(z_0,I)$ is the harmonic measure of $I$ at $z_0$ with respect to $\mathbb H$, and $C_1,C_2$ are absolute constants.
\end{proposition}

We also use the following form of the Koebe distortion theorem on $\mathbb H$.

\begin{proposition}\label{|Fw-Fz|:label}
Let $F\colon \mathbb H\to\mathbb C$ be conformal. Then
\[
|F(w)-F(z)|\ge \frac{1}{4}\,\tanh d_{\mathbb H}(z,w)\; d_F(z)
\]
for all $z,w\in\mathbb H$.
\end{proposition}

\begin{proof}
Fix $z,w\in\mathbb H$. Put $\xi=\kappa^{-1}(T_z(w))$ and $r=|\xi|$, so that $\Theta_z(\xi)=F(w)$ and
$d_{\mathbb D}(0,\xi)=d_{\mathbb H}(z,w)$. The growth theorem \cite[Theorem 1.3]{Pom} gives
\[
|F(w)-F(z)|=|\Theta_z(\xi)-\Theta_z(0)|\ge|\Theta_z'(0)|\,\frac{r}{(1+r)^2}.
\]
Since $2(1+r^2)-(1+r)^2=(1-r)^2\ge0$,
\[
\frac{r}{(1+r)^2}\ge\frac14\cdot\frac{2r}{1+r^2}=\frac14\tanh d_{\mathbb D}(0,\xi)=\frac14\tanh d_{\mathbb H}(z,w).
\]
On the other hand, the Schwarz lemma (see \cite[Corollary~1.4]{Pom}) gives
\[
d_F(z)=\operatorname{dist}\bigl(\Theta_z(0),\partial \Theta_z(\mathbb D)\bigr)\le|\Theta_z'(0)|.
\]
Combining these inequalities proves the proposition.
\end{proof}

We now adapt the proof of the twist theorem to graph domains. The adaptation consists of the following two lemmas.
For a conformal map $F$ on $\mathbb H$ (extended to $\overline{\mathbb H}$ when this is possible), we set
\[
\widetilde \phi(z,\zeta)=
\begin{cases}
\dfrac{F(z)-F(\zeta)}{z-\zeta}, & z\neq\zeta,\\[2mm]
F'(z), & z=\zeta.
\end{cases}
\]

\begin{lemma}\label{arg_def:label}
Let $\Omega=\Omega^+$ be the graph domain above the graph $\Gamma$ of a continuous function, and let
$F\colon \mathbb H\to\Omega$ be conformal with $F(\infty)=\infty$.
Then there exists a continuous branch $\phi$ of $\arg\widetilde\phi$ on $\mathbb H\times\overline{\mathbb H}$ such that
\[
|\phi(z,\zeta)|< \frac{3\pi}{2}, \quad (z,\zeta)\in \mathbb{H}\times \mathbb{R}.
\]
Moreover, $\phi(z,\zeta)=\phi(\zeta,z)$ for all $(z,\zeta)\in\mathbb H\times\mathbb H$, and $\phi(z,z)$ is a continuous branch of $\arg F'(z)$.
\end{lemma}

\begin{proof}
The function $\widetilde\phi$ is continuous on $\mathbb H\times\overline{\mathbb H}$, because $F$ is holomorphic on $\mathbb H$,
continuous on $\overline{\mathbb H}$, and $z\ne\zeta$ whenever $\zeta\in\mathbb R$.
It does not vanish, since $F'\ne0$ on $\mathbb H$ and $F$ is injective on $\overline{\mathbb H}$.
As $\mathbb H\times\overline{\mathbb H}$ is convex, $\widetilde\phi$ admits continuous branches of the argument, and any two of them differ by a constant in $2\pi\mathbb Z$.

Let $(z,\zeta)\in\mathbb H\times\mathbb R$. Then $F(z)\in\Omega$ and $F(\zeta)\in\Gamma$. The vector $F(z)-F(\zeta)$ is nonzero, and it cannot point vertically downward,
since the points directly below a point of $\Gamma$ do not belong to $\Omega$. Hence there is a continuous function $\alpha$ on $\mathbb H\times\mathbb R$ with values in $(-\pi/2,3\pi/2)$ such that $\alpha(z,\zeta)$ is an argument of $F(z)-F(\zeta)$. Let $\beta(z,\zeta)\in(0,\pi)$ be the argument of $z-\zeta$.
Then $\alpha-\beta$ is a continuous branch of $\arg\widetilde\phi$ on the connected set $\mathbb H\times\mathbb R$, with values in $(-3\pi/2,3\pi/2)$.
We choose $\phi$ as the continuous branch on $\mathbb H\times\overline{\mathbb H}$ that coincides with $\alpha-\beta$ at one point of $\mathbb H\times\mathbb R$. Then $\phi=\alpha-\beta$ on all of $\mathbb H\times\mathbb R$, which gives the bound.

Since $\widetilde\phi(z,\zeta)=\widetilde\phi(\zeta,z)$, the function $(z,\zeta)\mapsto\phi(\zeta,z)$ is also a continuous branch of $\arg\widetilde\phi$ on $\mathbb H\times\mathbb H$.
It coincides with $\phi$ on the diagonal, and therefore on the connected set $\mathbb H\times\mathbb H$. The last assertion is clear from $\widetilde\phi(z,z)=F'(z)$.
\end{proof}

\begin{lemma}\label{log_dis_est:label}
Let $F\colon \mathbb H\to\mathbb C$ be conformal, and let $\phi$ be a continuous branch of $\arg\widetilde\phi$ on $\mathbb H\times\mathbb H$.
Then
\[
\left|\phi(w,a)-\phi(a,a)\right|\le 4\, d_{\mathbb H}(w,a)+\frac{\pi}{2}
\]
for all $a,w\in\mathbb H$.
\end{lemma}

\begin{proof}
Fix $a\in\mathbb H$. The function
\[
q(w)=\frac{F(w)-F(a)}{(w-a)F'(a)}\quad(w\ne a),\qquad q(a)=1,
\]
is holomorphic and zero-free on $\mathbb H$, and $w\mapsto\phi(w,a)-\phi(a,a)$ is the continuous branch of $\arg q$ on $\mathbb H$ that vanishes at $w=a$.

Write $w=T_a^{-1}(\kappa(\xi))$ with $\xi\in\mathbb D$, so that $\xi=0$ corresponds to $w=a$ and $d_{\mathbb D}(0,\xi)=d_{\mathbb H}(w,a)$.
Since $w-a=(\operatorname{Im}a)\bigl(\kappa(\xi)-i\bigr)=(\operatorname{Im}a)\,\frac{2i\xi}{1-\xi}$ and $\Theta_a'(0)=2i(\operatorname{Im}a)F'(a)$, a direct computation gives
\begin{equation}\label{disk_identity}
q(w)
=(1-\xi)\,
\frac{\Theta_a(\xi)-\Theta_a(0)}{\xi \Theta_a'(0)}.
\end{equation}
Let $\log\frac{\Theta_a(\xi)-\Theta_a(0)}{\xi \Theta_a'(0)}$ be the branch that vanishes at $\xi=0$, and let $\log(1-\xi)$ be the principal branch.
By \eqref{disk_identity}, the imaginary part of the sum of these two logarithms, regarded as a function of $w$, is a continuous branch of $\arg q$ that vanishes at $w=a$.
By uniqueness of such branches on the connected set $\mathbb H$, it equals $\phi(w,a)-\phi(a,a)$.
Proposition~\ref{fz-f0/f':label} applied to $\Theta_a$ and the inequality $|\arg(1-\xi)|<\pi/2$, which holds because $\operatorname{Re}(1-\xi)>0$, now give
\[
\left|\phi(w,a)-\phi(a,a)\right|
\le \left|\log\frac{\Theta_a(\xi)-\Theta_a(0)}{\xi \Theta_a'(0)}\right|+\left|\arg (1-\xi)\right|
\le 4\, d_{\mathbb D}(0,\xi)+\frac{\pi}{2}=4\, d_{\mathbb H}(w,a)+\frac{\pi}{2}.
\qedhere
\]
\end{proof}

We are now ready to prove the main step, following the proof of \cite[Theorem 6.18]{Pom}.
It suffices to consider the domain $\Omega^+$ above a graph. Indeed, if $F\colon\mathbb H\to\Omega^-$ is conformal with $F(\infty)=\infty$, where $\Omega^-$ lies below the graph of $f$, then by the rotation $z\mapsto -z$ the map $-F$ sends $\mathbb H$ conformally onto the domain above the graph of $x\mapsto -f(-x)$, and $\log(-F)'=\log F'+i\pi$.

\begin{theorem}\label{argFprime-UHP}
Let $\Omega=\Omega^+$ be the graph domain above the graph of a continuous function, and let $F\colon \mathbb H\to\Omega$ be conformal with $F(\infty)=\infty$.
Then every continuous branch of $\arg F'$ is bounded on $\mathbb H$.
\end{theorem}

\begin{proof}
Any two continuous branches of $\arg F'$ differ by a constant, so it suffices to show that $\phi(z,z)$ is bounded on $\mathbb H$, where $\phi$ is the function given by Lemma~\ref{arg_def:label}.

Fix $z=x+iy\in\mathbb H$ and let $I=[x-y,x+y]$. The segment $I$ subtends the angle $\pi/2$ at $z$, so $\omega_{\mathbb H}(z,I)=1/2$.
By Proposition~\ref{uhp-f(S):label}, there exists a hyperbolic geodesic segment $S\subset\mathbb H$
joining $z$ to some point $\zeta\in I$ such that
\begin{equation}\label{FS/DF}
\ell(F(S))\le C_1 e^{2C_2}\, d_F(z).
\end{equation}
Since $S$ has infinite hyperbolic length, there is a point $z'\in S$ with $d_{\mathbb H}(z,z')=1$. Let $S'$ be the subarc of $S$ from $z'$ to $\zeta$.

By the symmetry of $\phi$ and Lemma~\ref{log_dis_est:label} with $(w,a)=(z',z)$,
\begin{equation}\label{|phi|_1}
|\phi(z,z)-\phi(z,z')|=|\phi(z',z)-\phi(z,z)|
\le 4\, d_{\mathbb H}(z,z')+\frac{\pi}{2}
=4+\frac{\pi}{2}.
\end{equation}

Next we estimate the change of $w\mapsto\phi(z,w)$ along $S'$. Let $\alpha(w)$ and $\beta(w)$ be continuous branches of $\arg(F(z)-F(w))$ and $\arg(z-w)$ along $S'$.
Then $\alpha-\beta$ differs from $\phi(z,\cdot)$ by a constant on $S'$, and hence
\begin{equation}\label{|phi|_21}
|\phi(z,z')-\phi(z,\zeta)|
\le|\alpha(z')-\alpha(\zeta)|+|\beta(z')-\beta(\zeta)|.
\end{equation}
The geodesic $S$ lies on a circle orthogonal to $\mathbb R$ or on a vertical line passing through $z$. In either case the directions of $w-z$, $w\in S\setminus\{z\}$, lie in a closed half-plane whose boundary line passes through $0$, and they lie in its interior unless they are all equal. Therefore $|\beta(z')-\beta(\zeta)|<\pi$.
For the first term, note that $d\alpha=\operatorname{Im}\bigl(F'(w)\,dw/(F(w)-F(z))\bigr)$ on $S'\setminus\{\zeta\}$. Since the hyperbolic distance from $z$ increases along $S$, every $w\in S'$ satisfies $d_{\mathbb H}(z,w)\ge 1$, and Proposition~\ref{|Fw-Fz|:label} gives $|F(w)-F(z)|\ge\frac14(\tanh1)\,d_F(z)$. As $\alpha$ is continuous at $\zeta$, we deduce from \eqref{FS/DF} that
\begin{equation}\label{|phi|_22}
|\alpha(z')-\alpha(\zeta)|
\le
\int_{S'}\frac{|F'(w)|}{|F(w)-F(z)|}\,|dw|
\le
\frac{4}{d_F(z)\tanh1}\,\ell(F(S))
\le
\frac{4C_1 e^{2C_2}}{\tanh1}.
\end{equation}

Finally, $|\phi(z,\zeta)|<3\pi/2$ by Lemma~\ref{arg_def:label}, because $\zeta\in\mathbb R$.
Together with \eqref{|phi|_1}, \eqref{|phi|_21}, and \eqref{|phi|_22}, this shows that
\[
|\phi(z,z)|\le \frac{3\pi}{2}+4+\frac{\pi}{2}+\pi+\frac{4C_1e^{2C_2}}{\tanh1},
\]
and the right-hand side does not depend on $z$.
\end{proof}

The following standard fact completes the argument.

\begin{theorem}\label{bmoa:label}
Let $F$ be holomorphic on $\mathbb H$ with $F'(z)\neq0$ for all $z\in\mathbb H$.
If a continuous branch of $\arg F'$ is bounded on $\mathbb H$, then
$\log F' \in \mathrm{BMOA}(\mathbb H)$.
\end{theorem}

\begin{proof}
Let $\Phi=\log F'$ be the holomorphic branch whose imaginary part $V$ is the given bounded branch of $\arg F'$.
By the Cauchy--Riemann equations, $|\Phi'|^2=|\nabla V|^2$ on $\mathbb H$. The bounded harmonic function $V$ is the Poisson integral of its non-tangential boundary values $V^*\in L^\infty(\mathbb R)$ (see \cite[Chapter I, Theorem 3.5]{Ga}).
Since $L^\infty(\mathbb R)\subset\mathrm{BMO}(\mathbb R)$, the theorem of Fefferman and Stein (see \cite[Chapter VI]{Ga}), which states that $|\nabla P[v](z)|^2\,y\,dx\,dy$ is a Carleson measure for the Poisson integral $P[v]$ of every $v\in\mathrm{BMO}(\mathbb R)$, shows that
$|\Phi'(z)|^2\,y\,dx\,dy=|\nabla V(z)|^2\,y\,dx\,dy$ is a Carleson measure. Hence $\Phi\in\mathrm{BMOA}(\mathbb H)$.
\end{proof}

As a consequence of Theorems~\ref{argFprime-UHP} and \ref{bmoa:label}, together with the reduction preceding Theorem~\ref{argFprime-UHP}, we obtain the following.

\begin{corollary}\label{graph_BMO:label}
Every graph domain is a $\mathrm{BMO}$ domain.
\end{corollary}

A geometric characterization of $\mathrm{BMO}$ domains was given by Bishop and Jones \cite{BJ}
(see also \cite[Theorem 5.3]{GM}).
A quasicircle passing through $\infty$ is called a \emph{Bishop--Jones quasicircle} if one of its complementary domains is a $\mathrm{BMO}$ domain.
We note that the boundary of a $\mathrm{BMO}$ domain need not be a quasicircle. For instance, the graph of $f(x)=\sqrt{|x|}$ is not a quasicircle, while both of its complementary domains are $\mathrm{BMO}$ domains by Corollary~\ref{graph_BMO:label}.

\section{Graphs of Weierstrass--Zygmund functions and asymptotic conformality}\label{Weierstrass_Zygmund}

In this section we relate the Zygmund regularity of a function $f$ to the geometry of its graph.
We show that the graph of a Zygmund function is a quasicircle and that the graph of a little Zygmund function is an asymptotically conformal quasicircle.
Finally, we verify that a Weierstrass-type function provides an example that is not asymptotically smooth.

Throughout this section, for a continuous function $f\colon\mathbb{R}\to\mathbb{R}$ we put
\[
F(x)=x+i f(x),\qquad \Gamma=F(\mathbb R)=\{x+i f(x)\colon x\in\mathbb{R}\}.
\]

It is known that, for every $f\in\Lambda_*(\mathbb R)$,
the natural parametrization $F(x)=x+if(x)$ is a quasisymmetric
embedding; see \cite[Theorem 1.4 and Remark 3.1]{KO}.
We include a direct proof based on
Proposition~\ref{little-strongly}.

\begin{theorem}\label{zyg-qs:label}
If $f\in\Lambda_*(\mathbb R)$, then $F(x)=x+if(x)$ is a quasisymmetric embedding of $\mathbb{R}$ into $\mathbb C$.
\end{theorem}

\begin{proof}
Let $M=\|f\|_{\Lambda_*}$. Fix distinct $x,a,b\in\mathbb R$ with $|x-a|\le|x-b|$, and put
$D(u)=f(x+u)-f(x)$, $t=|x-a|$ and $s=|x-b|$. Write $a=x+\sigma t$ and $b=x+\tau s$
with $\sigma,\tau\in\{-1,1\}$.
The affine function interpolating $f$ at $x$ and $x+\sigma s$ takes the
value $f(x)+\frac ts D(\sigma s)$ at $a$, so Proposition~\ref{little-strongly}
applied on the interval with endpoints $x$ and $x+\sigma s$ gives
\[
\Bigl|D(\sigma t)-\frac ts\,D(\sigma s)\Bigr|\le\frac{M s}{2},
\qquad\text{hence}\qquad
|D(\sigma t)|\le|D(\sigma s)|+\frac{M s}{2}.
\]
Since $D(s)+D(-s)=\Delta_s^2f(x)$, we also have
$|D(\sigma s)|\le|D(\tau s)|+M s$. Consequently,
\[
|F(a)-F(x)|\le t+|f(a)-f(x)|\le\Bigl(1+\frac{3M }{2}\Bigr)s+|f(b)-f(x)|
\le\Bigl(2+\frac{3M }{2}\Bigr)|F(b)-F(x)|,
\]
because $|F(b)-F(x)|\ge\max\{s,\,|f(b)-f(x)|\}$.
Thus $F$ is weakly quasisymmetric, and therefore quasisymmetric
(see Section~\ref{sec:functions_embeddings_and_curves}).
\end{proof}

\begin{remark}
The converse of the fact that graphs of Zygmund functions are quasicircles is false.
Let $f$ be the continuous function with $f(x)=1$ for $x\le 1$ such that, for each integer $n\ge1$,
\[
f(x)=\begin{cases}
n+n(x-n), & x\in (n,n+\tfrac{1}{n}],\\
n+1, & x\in (n+\tfrac{1}{n},n+1].
\end{cases}
\]
Since $f$ is nondecreasing and locally Lipschitz, for $a<b$ the length of the graph over $[a,b]$ is at most
$(b-a)+(f(b)-f(a))\le\sqrt2\,|F(b)-F(a)|$. Thus the graph $\Gamma$ is a chord-arc curve, and hence a quasicircle.
However, for $n\ge3$ and $h=1/n$ we have $\Delta_h^2f(n)=nh$ and
$|F(n+h)-F(n)|/|F(n)-F(n-h)|=\sqrt{n^2+1}$. Hence $f\notin\Lambda_*(\mathbb R)$,
and $F$ is not even weakly quasisymmetric.
This example also shows that the assertion in a problem of \cite[p.~261]{GM} does not hold.
\end{remark}

\begin{theorem}\label{graph-AS}
If $f\in\lambda_*(\mathbb R)$, then $F(x)=x+if(x)$ is an asymptotically symmetric embedding, and $\Gamma=F(\mathbb R)$ is an asymptotically conformal quasicircle.
\end{theorem}

\begin{proof}
By Proposition~\ref{little-strongly}, $f\in\lambda_*^S(\mathbb R)$. Hence $F$ is an
asymptotically symmetric embedding of $\mathbb{R}$ into $\mathbb{C}$ by \cite[Lemma 5.1]{MT2},
and in particular $\Gamma$ is a quasicircle.
Since $|x-y|\le|F(x)-F(y)|$ for all $x,y\in\mathbb R$, the inverse map $F^{-1}\colon\Gamma\to\mathbb R$
is uniformly continuous, and \cite[Theorem 5.4]{MT1} shows that $\Gamma$ is an
asymptotically conformal quasicircle.
\end{proof}
\begin{remark}
A weakly asymptotically symmetric self-homeomorphism of $\mathbb S$ is asymptotically
symmetric by \cite[Theorem 4.1]{BY}, and whether the two notions coincide in general is
asked in \cite[Section 4]{BY}. Some affirmative answers are given in \cite{MT3}, and they
yield another proof of the first assertion of Theorem~\ref{graph-AS}. Let $f\in\lambda_*(\mathbb R)$, fix $x\in\mathbb R$
and $h>0$, put $\Delta_+=f(x+h)-f(x)$ and $\Delta_-=f(x)-f(x-h)$, and let
$\varepsilon(h)=\sup_{y\in\mathbb R}|\Delta_h^2f(y)|/h$. Since $\Delta_+-\Delta_-=\Delta_h^2f(x)$, we have
$|\Delta_++\Delta_-|\le 2|\Delta_-|+\varepsilon(h)h$, and hence
\[
\frac{|F(x+h)-F(x)|^2}{|F(x)-F(x-h)|^2}-1
=\frac{(\Delta_+-\Delta_-)(\Delta_++\Delta_-)}{h^2+\Delta_-^2}
\le\frac{\varepsilon(h)h\,\bigl(2|\Delta_-|+\varepsilon(h)h\bigr)}{h^2+\Delta_-^2}
\le\varepsilon(h)+\varepsilon(h)^2 .
\]
The same bound holds for the reciprocal ratio. Thus the distortion of $F$ on symmetric
triples tends to $1$ uniformly as $h\to0$, and since $F$ is quasisymmetric by
Theorem~\ref{zyg-qs:label}, \cite[Theorem 1.1]{MT3} shows that $F$ is asymptotically symmetric.
\end{remark}

We now show that a Weierstrass-type function belongs to the little Zygmund class.

\begin{theorem}\label{weierstrass:label}
The function
\[
f(x)=\sum_{n=1}^{\infty}\frac{1}{2^n\sqrt{n}}\cos(2^n x)
\]
belongs to $\lambda_*(\mathbb R)$.
\end{theorem}

\begin{proof}
Write $a_n=2^{-n}n^{-1/2}$. Since $\sum_n a_n<\infty$, the series converges absolutely
and uniformly, so $f$ is a bounded continuous $2\pi$-periodic function.
The series is lacunary with frequencies $2^n$, and $2^na_n=n^{-1/2}\to0$.
By \cite[Vol.~I, Chapter II, Theorem~(4.10), p.~47]{Zy}, $f$ belongs to the little Zygmund class
of $2\pi$-periodic functions, that is,
$\sup_{x}|\Delta_h^2f(x)|/|h|\to0$ as $h\to0$.
In particular, there is $\delta>0$ such that $|\Delta_h^2f(x)|\le|h|$ for all $x$ and $0<|h|<\delta$,
while $|\Delta_h^2f(x)|\le4\|f\|_\infty\le(4\|f\|_\infty/\delta)|h|$ for $|h|\ge\delta$.
Hence $f\in\Lambda_*(\mathbb R)$, and therefore $f\in\lambda_*(\mathbb R)$.
\end{proof}







\begin{corollary}\label{newexample:label}
For the function $f$ of Theorem~\ref{weierstrass:label}, the graph $\Gamma$ of $f$
is an asymptotically conformal Bishop--Jones quasicircle that is not asymptotically smooth.
\end{corollary}

\begin{proof}
By Theorems~\ref{weierstrass:label} and \ref{graph-AS}, $\Gamma$ is an asymptotically conformal quasicircle, and it is a Bishop--Jones quasicircle by Corollary~\ref{graph_BMO:label}.
The coefficients satisfy $\sum_n\bigl(2^n\cdot2^{-n}n^{-1/2}\bigr)^2=\sum_n1/n=\infty$. By a classical theorem on lacunary series, $f$ is then differentiable only on a set of Lebesgue measure zero
(see \cite[Vol.~I, p.~206]{Zy}).
Since a function of bounded variation on an interval is differentiable almost everywhere there, $f$ is not of bounded variation on any interval.
Hence $\Gamma$ is not locally rectifiable, and in particular it is not asymptotically smooth.
\end{proof}

\begin{corollary}\label{strict}
The inclusions
\[
\mathrm{VMOA}(\mathbb{H})
\subset
\mathrm{BMOA}(\mathbb{H})\cap B_0(\mathbb{H}),\qquad
\mathrm{SS}(\mathbb{R})
\subset
\mathrm{SQS}(\mathbb{R})\cap \mathrm{S}(\mathbb{R})
\]
are strict.
\end{corollary}

\begin{proof}
Let $\Omega$ be the domain above the graph $\Gamma$ of the function $f$
of Theorem~\ref{weierstrass:label}, and let $G\colon\mathbb H \to \Omega$ be
a conformal mapping with $G(\infty)=\infty$.
Since $f$ is bounded, the boundary extension $g\colon\mathbb R \to
\Gamma$ of $G$ and its inverse $g^{-1}$ are uniformly continuous by \cite[Lemma 4.1]{MT2}.

By Corollary~\ref{newexample:label},
$\Gamma$ is an asymptotically conformal quasicircle that is
not asymptotically smooth. By
the half-plane version of
Pommerenke's characterization \cite{Po78} established in \cite[Theorem 3.1, Corollary 7.4]{MT1}, $\log G'$ belongs to
$B_0(\mathbb{H}) \setminus \mathrm{VMOA}(\mathbb{H})$. Since $\log G'\in\mathrm{BMOA}(\mathbb{H})$
by Corollary~\ref{graph_BMO:label}, the first inclusion is strict.

Let $\Omega^*$ be the domain below $\Gamma$, let $G^*\colon\mathbb H^*\to\Omega^*$ be a conformal mapping with $G^*(\infty)=\infty$, and let $g^*\colon\mathbb R\to\Gamma$ be its boundary extension.
The map $z\mapsto -G^*(-z)$ sends $\mathbb H$ conformally onto the domain above the graph of the bounded function $x\mapsto -f(-x)$, so \cite[Lemma 4.1]{MT2} shows that $g^*$ and $(g^*)^{-1}$ are also uniformly continuous.
Let
\[
h=g^{-1}\circ g^*\colon\mathbb R\to\mathbb R
\]
be the conformal welding of $\Gamma$, with the convention of Section~\ref{intro}, which is the one used in \cite{MT1}.
Since $\log G'\in B_0(\mathbb H)$, \cite[Corollary 5.7]{MT1} gives $h\in\mathrm{S}(\mathbb R)$, and since $\log G'\notin\mathrm{VMOA}(\mathbb H)$, \cite[Corollary 7.10]{MT1} gives $h\notin\mathrm{SS}(\mathbb R)$.
Moreover, $\log G'\in\mathrm{BMOA}(\mathbb H)$ gives $h\in\mathrm{SQS}(\mathbb R)$ by \cite[Theorem 7.2]{Sh}.
Hence $h\in\bigl(\mathrm{SQS}(\mathbb{R})\cap \mathrm{S}(\mathbb{R})\bigr) \setminus \mathrm{SS}(\mathbb{R})$, and the second inclusion is strict.
\end{proof}

\section{A localized Lacunary construction}\label{sec:construction_of_the_lacunary_series}

In this section we refine Corollary~\ref{newexample:label} by constructing a function
whose graph is an asymptotically conformal chord-arc curve that is not asymptotically smooth.

If $f$ is locally absolutely continuous and $f'\in\mathrm{BMO}(\mathbb{R})$, then the graph $\Gamma$ of $f$ is a chord-arc curve. A simple proof for graphs is
given in \cite[Lemma 5.6]{MT2}. If moreover $f \in \lambda_*(\mathbb{R})$, then
Theorem~\ref{graph-AS} shows that $\Gamma$ is an asymptotically conformal quasicircle.

\begin{theorem}\label{thm:main}
There exists a compactly supported locally absolutely continuous function $f\in\lambda_*(\mathbb{R})$
such that
\[
  f'\in\mathrm{BMO}(\mathbb{R})\setminus\mathrm{VMO}(\mathbb{R}).
\]
\end{theorem}

The proof occupies the following four subsections.
Along the way we also show that the graph of $f$ is not asymptotically smooth. Note that $f' \in \mathrm{VMO}(\mathbb{R})$
is a sufficient condition for the graph to be asymptotically smooth
\cite[Lemma 5.6]{MT2}, so this is a stronger statement than $f'\notin\mathrm{VMO}(\mathbb R)$.

\subsection{Definition of the function $f$}

For $N\ge1$, define
\[
  W_N(t)\coloneqq\sum_{k=1}^{N}4^{-k}\sin(2\pi4^kt),\qquad
  S_N(t)\coloneqq\sum_{k=1}^{N}\cos(2\pi4^kt),
\]
so that $W_N'=2\pi S_N$.
Fix a cutoff function $\chi\in C_c^\infty\bigl(-\tfrac34,\tfrac34\bigr)$ such that
$0\leq\chi\leq1$ and $\chi(t)=1$ for $|t|\leq\tfrac12$.
For $n\ge2$, set
\[
  x_n\coloneqq 2^{-n},
  \qquad
  \ell_n\coloneqq 2^{-3n},
  \qquad
  N_n\coloneqq n^2,
  \qquad
  \varepsilon_n\coloneqq \tfrac1n,
\]
and define
\[
  f_n(x)
  \coloneqq\varepsilon_n\ell_n
    \chi\Bigl(\frac{x-x_n}{\ell_n}\Bigr)
    W_{N_n}\Bigl(\frac{x-x_n}{\ell_n}\Bigr),
  \qquad
  f(x)\coloneqq\sum_{n=2}^{\infty}f_n(x).
\]

The support of $f_n$ is contained in
$J_n\coloneqq \bigl[x_n-\tfrac34\ell_n,\,x_n+\tfrac34\ell_n\bigr]$.
These intervals are pairwise disjoint, because
$x_n-x_{n+1}=2^{-n-1}$ while
$\tfrac34(\ell_n+\ell_{n+1})=\tfrac{27}{32}\cdot2^{-3n}<2^{-n-1}$.
More precisely, for $n\ge2$,
\begin{equation}\label{eq:gaps}
  |J_n|=\tfrac32\cdot 2^{-3n},\qquad\operatorname{dist}(J_n,J_{n+1})\geq 2^{-n}\Bigl(\tfrac12-\tfrac{27}{32}\cdot2^{-2n}\Bigr)\ge\tfrac14\cdot 2^{-n}.
\end{equation}
The intervals $J_n$ accumulate only at the origin.
Since $|W_N|\leq \sum_{k=1}^\infty 4^{-k} = 1/3$, we have $\|f_n\|_\infty\leq \varepsilon_n\ell_n/3$, so the series defining $f$ converges uniformly.
In particular, $f$ is continuous, compactly supported, and $f(0)=0$.

\subsection{Auxiliary estimates}

\begin{claim}\label{lem:WN-Zygmund}
For every $N\ge1$,
\[
  \|W_N\|_{\infty}+\|W_N\|_{\Lambda_*}\leq 106.
\]
\end{claim}

\begin{proof}
We already know that $\|W_N\|_\infty\leq 1/3$.
The identity
\[
\Delta_s^2\sin(\lambda t)=2\sin(\lambda t)\bigl(\cos(\lambda s)-1\bigr)
\]
yields
\[
  \bigl|\Delta_s^2\bigl(4^{-k}\sin(2\pi4^kt)\bigr)\bigr|\leq 4\pi^2 4^{-k}\min\{(4^k|s|)^2,1\}.
\]
For $0<|s|<1$, splitting the sum according to whether
$4^k|s|\leq 1$ or $4^k|s|>1$ gives
\[
  |\Delta_s^2W_N(t)|\leq 4\pi^2|s|^2\sum_{4^k|s|\leq 1}4^k+4\pi^2\sum_{4^k|s|>1}4^{-k} \leq \frac{32\pi^2}{3}|s|.
\]
For $|s|\geq 1$, we simply have
$|\Delta_s^2W_N(t)|\leq 4\|W_N\|_\infty\leq \frac{4}{3}|s|$.
Hence
$\|W_N\|_{\infty}+\|W_N\|_{\Lambda_*}\leq (1+32\pi^2)/3<106$.
\end{proof}

We shall multiply $W_N$ by a fixed smooth cutoff, and the following multiplier estimate for the Zygmund class will be used.

\begin{lemma}\label{lem:Zygmund-multiplier}
Let $\varphi\in C_c^\infty(\mathbb{R})$.
There is a constant $C_\varphi$ depending only on $\varphi$ such that for every bounded $u\in\Lambda_*(\mathbb{R})$,
\[
  \|\varphi u\|_{\Lambda_*}\leq C_\varphi\bigl(\|u\|_\infty+\|u\|_{\Lambda_*}\bigr).
\]
\end{lemma}

\begin{proof}
We first recall a standard consequence of the Zygmund condition. If $0<|h|\leq 1/2$, then
\begin{equation}\label{eq:first-difference}
  |u(x+h)-u(x)|\leq 4|h|\bigl(1+\log(1/|h|)\bigr)\bigl(\|u\|_\infty+\|u\|_{\Lambda_*}\bigr).
\end{equation}
Indeed, set $D(h)=u(x+h)-u(x)$ and choose an integer $m\geq 0$ such that $1/2\leq 2^m|h|<1$.
Iterating the identity $\Delta_h^2u(x+h)=D(2h)-2D(h)$ yields
\begin{align*}
  |D(h)|&=\Bigl|2^{-m}D(2^mh)-\sum_{j=0}^{m-1}2^{-j-1}\Delta_{2^jh}^2u(x+2^jh)\Bigr|
  \leq 4|h|\,\|u\|_{\infty}+\frac{m}{2}\,|h|\,\|u\|_{\Lambda_*}\\
  &\leq |h|\left(4\|u\|_{\infty}+\frac{\log (1/|h|)}{2\log 2}\,\|u\|_{\Lambda_*}\right),
\end{align*}
which implies \eqref{eq:first-difference}.

For $0<|h|\leq 1/2$, the product rule for second differences,
\begin{align*}
\Delta_h^2(\varphi u)(x)={}&\varphi(x)\Delta_h^2u(x)+\bigl(\varphi(x+h)-\varphi(x)\bigr)\bigl(u(x+h)-u(x)\bigr)\\
&+\bigl(\varphi(x-h)-\varphi(x)\bigr)\bigl(u(x-h)-u(x)\bigr)+u(x)\Delta_h^2\varphi(x),
\end{align*}
together with \eqref{eq:first-difference} gives
\begin{align*}
  |\Delta_h^2(\varphi u)(x)|
     \leq & \|\varphi\|_{\infty}\|u\|_{\Lambda_*}|h|+\|\varphi'\|_{\infty}|h|\bigl(|D(h)|+|D(-h)|\bigr)+\|u\|_{\infty}\|\varphi''\|_{\infty}|h|^2\\
     \leq & C_\varphi |h|\bigl(\|u\|_\infty+\|u\|_{\Lambda_*}\bigr),
\end{align*}
because $t\mapsto t\bigl(1+\log(1/t)\bigr)$ is bounded on $(0, 1/2]$.
For $|h|>1/2$,
$|\Delta_h^2(\varphi u)(x)|\leq 4\|\varphi\|_\infty \|u\|_\infty < 8\|\varphi\|_\infty \|u\|_\infty|h|$.
Taking the supremum over all $x$ and $h \neq 0$ completes the proof.
\end{proof}

The next two claims concern the lacunary sums $S_N$.

\begin{claim}\label{lem:lacunary-L1}
For every $N\geq 1$,
\[
  \int_{-1/2}^{1/2}|S_N(t)|\,dt\geq \frac{\sqrt N}{\sqrt{6}}.
\]
\end{claim}

\begin{proof}
By orthogonality,
\begin{equation}\label{eq:S-L2}
  \|S_N\|_{L^2(-1/2,1/2)}^2=\frac N2.
\end{equation}
We next show that
\begin{equation}\label{eq:S-L4}
  \|S_N\|_{L^4(-1/2,1/2)}^4\leq \frac34 N^2.
\end{equation}
Writing $\cos(2\pi4^kt)=\frac12\bigl(e^{2\pi i4^kt}+e^{-2\pi i4^kt}\bigr)$, we expand
\[
  S_N(t)^4=\frac1{16}\sum_{\substack{1\leq k_1,k_2,k_3,k_4\le N\\
  \sigma_1,\sigma_2,\sigma_3,\sigma_4\in\{-1,1\}}}e^{2\pi i(\sigma_14^{k_1}+\sigma_24^{k_2}+\sigma_34^{k_3}+\sigma_44^{k_4})t}.
\]
A term has nonzero integral over $[-1/2,1/2]$ only if
$\sigma_14^{k_1}+\sigma_24^{k_2}+\sigma_34^{k_3}+\sigma_44^{k_4}=0$, in which case the integral equals $1$.
By the uniqueness of base-$4$ expansions, a sum of two or three powers of $4$ (repetitions allowed) is never a power of $4$, and
$4^{p}+4^{q}=4^{r}+4^{s}$ holds only if $\{p,q\}=\{r,s\}$ as multisets.
Hence exactly two of the $\sigma_j$ are equal to $1$, and for each of the $\binom42=6$ choices of their positions,
the indices attached to the positive signs can be chosen in at most $N^2$ ways, after which the other two indices are determined up to order.
The number of terms with nonzero integral is therefore at most $12N^2$, which proves \eqref{eq:S-L4}.

By H\"older's inequality,
$\|S_N\|_{L^2}\leq \|S_N\|_{L^1}^{1/3}\|S_N\|_{L^4}^{2/3}$ on $(-1/2,1/2)$. Combined with \eqref{eq:S-L2} and \eqref{eq:S-L4}, this gives
\[
  \|S_N\|_{L^1(-1/2,1/2)}\geq\frac{\|S_N\|_{L^2}^3}{\|S_N\|_{L^4}^2}\geq \frac{(N/2)^{3/2}}{(\sqrt3/2)\,N}=\frac{\sqrt N}{\sqrt{6}}.
  \qedhere
\]
\end{proof}

\begin{claim}\label{lem:lacunary-BMO}
For every $N\geq1$,
\[
  \|S_N\|_{\mathrm{BMO}(\mathbb{R})}\leq 12\sqrt N.
\]
\end{claim}

\begin{proof}
Since the mean oscillation over $I$ is bounded by $\bigl(|I|^{-1}\int_I|S_N-c|^2\,dt\bigr)^{1/2}$ for every constant $c$, it suffices to bound this quantity for a suitable $c$.
Let $I$ be an interval of length $r>0$. For $1\leq k<j$,
\[
\left|\int_I\cos(2\pi4^kt)\cos(2\pi4^jt)\,dt\right|
\leq\frac{1}{2\pi}\left(\frac{1}{4^j-4^k}+\frac{1}{4^j+4^k}\right)\leq\frac{7}{6\pi}\,4^{-j}.
\]

Suppose first that $r\geq1$, and take $c=0$. Expanding the square and using $\sum_{j\ge1}(j-1)4^{-j}=1/9$, we see that
\[
\frac{1}{|I|}\int_I|S_N(t)|^2\,dt
\leq\frac N2+\frac{1}{4\pi r}\sum_{k=1}^N4^{-k}+\frac{7}{3\pi r}\sum_{1\leq k<j\leq N}4^{-j}
\leq\frac N2+\frac{1}{12\pi}+\frac{7}{27\pi} \le N.
\]

Now suppose that $0<r<1$, and choose the integer $m\geq0$ with $4^{-(m+1)}<r\leq4^{-m}$.
Split $S_N=L+H$, where $L$ is the sum over $k\leq\min\{m,N\}$ and $H$ the sum over $m<k\leq N$ (so $H=0$ if $m\ge N$).
For $x,y\in I$,
\[
|L(x)-L(y)|\leq 2\pi r\sum_{k=1}^{m}4^k\le\frac{2\pi r}{3}\,4^{m+1}\leq \frac{8\pi}{3}.
\]
Fix $x_0\in I$ and take $c=L(x_0)$. If $m\ge N$, then $|I|^{-1}\int_I|S_N-c|^2\le(8\pi/3)^2$. If $m<N$, then, bounding each
$\cos^2$ term in $H^2$ by $1$ and estimating the cross terms as above,
\begin{align*}
  \frac{1}{|I|}\int_I|S_N(t)-L(x_0)|^2\,dt&\leq\frac{2}{|I|}\int_I|L(t)-L(x_0)|^2\,dt+\frac{2}{|I|}\int_I|H(t)|^2\,dt\\
 & \leq 2\left(\frac{8\pi}{3}\right)^2+2\Bigl(N-m+\frac{7}{3\pi r}\sum_{m<k<j\leq N}4^{-j}\Bigr)\\
 &\leq 2\left(\frac{8\pi}{3}\right)^2 +2N+\frac{56}{27\pi},
\end{align*}
where we used $\sum_{m<k<j}4^{-j}\le 4^{-m}/9$ and $r>4^{-(m+1)}$. In both cases the right-hand side is at most $144N$, which completes the proof.
\end{proof}

Finally, we estimate the $\mathrm{BMO}$ norm of the product of a $\mathrm{BMO}$ function and a smooth cutoff.

\begin{lemma}\label{lem:BMO-multiplier}
Let $\varphi\in C_c^\infty(\mathbb{R})$, and let $J$ be a bounded interval containing
$\supp\varphi$.
Then, for every $u\in\mathrm{BMO}(\mathbb{R})$,
\[
  \|\varphi u\|_{\mathrm{BMO}(\mathbb{R})}\leq C_{\varphi,J}\left(\|u\|_{\mathrm{BMO}(\mathbb{R})} +\frac{1}{|J|}\int_J|u(x)|\,dx\right),
\]
where $C_{\varphi,J}$ depends only on $\varphi$ and $|J|$.
\end{lemma}

\begin{proof}
Let $I\subset\mathbb{R}$ be an interval. If $I\cap\supp\varphi=\varnothing$, then $\varphi u=0$ on $I$ and there is nothing to prove.

Suppose that $I\cap\supp\varphi\neq\varnothing$ and $|I|\leq 2|J|$.
Let $J^*$ be the interval concentric with $J$ of length $10|J|$. Then $I \subset J^*$ and $2|I|< |J^*|$, so
$|u_I - u_{J^*}| \leq c \log(|J^*|/|I|)\,\|u\|_{\mathrm{BMO}(\mathbb{R})}$ with an absolute constant $c$ (see \cite[Lemma 1.1 (a), p.~217]{Ga}),
and similarly $|u_J-u_{J^*}|\le c\log 10\,\|u\|_{\mathrm{BMO}(\mathbb{R})}$.
Using $\varphi u-\varphi_Iu_I=\varphi(u-u_I)+(\varphi-\varphi_I)u_I$ and $|\varphi-\varphi_I|\le|I|\,\|\varphi'\|_\infty$ on $I$, we find
\begin{align*}
\frac1{|I|}\int_I|\varphi u-(\varphi u)_I|\,dx
&\leq \frac2{|I|}\int_I|\varphi u-\varphi_I u_I|\,dx
\leq 2\|\varphi\|_\infty\|u\|_{\mathrm{BMO}(\mathbb{R})}+2|I|\,\|\varphi'\|_\infty\,|u_I|\\
&\leq 2\|\varphi\|_\infty\|u\|_{\mathrm{BMO}(\mathbb{R})}+2|I|\,\|\varphi'\|_\infty\bigl(|u_J|+|u_I-u_{J^*}|+|u_J-u_{J^*}|\bigr).
\end{align*}
Since $|u_J|\le|J|^{-1}\int_J|u|$, $|I|\le2|J|$, and $|I|\log(|J|/|I|)\le|J|$, the right-hand side is bounded by
$C_{\varphi,J}\bigl(\|u\|_{\mathrm{BMO}(\mathbb{R})}+|J|^{-1}\int_J|u|\bigr)$.

Suppose next that $|I|>2|J|$. Since $\supp\varphi\subset J$,
\[
  \frac1{|I|}\int_I|\varphi u-(\varphi u)_I|\,dx\leq\frac{2}{|I|}\int_{I\cap J}|\varphi u|\,dx
  \le\frac{\|\varphi\|_\infty}{|J|}\int_J|u|\,dx.
\]
Taking the supremum over all intervals $I$ completes the proof.
\end{proof}

\subsection{Absolute continuity and little Zygmund regularity}

Writing $t=(x-x_n)/\ell_n$, we have
\begin{equation}\label{eq:fn-prime}
  f_n'(x)
  =\varepsilon_n\bigl(\chi'(t)W_{N_n}(t)+2\pi\chi(t)S_{N_n}(t)\bigr).
\end{equation}
Since $S_N$ is $1$-periodic and its summands are mutually orthogonal on every interval of length $1$,
\begin{equation}\label{eq:S-L1-J}
  \int_{-3/4}^{3/4}|S_N(t)|\,dt\leq \sqrt{\frac32}\left(\int_{-1}^{1}|S_N(t)|^2\,dt\right)^{1/2}=\sqrt{\frac{3N}{2}}\le 2\sqrt N.
\end{equation}
Consequently, using $\varepsilon_n\le1$ and $\varepsilon_n\sqrt{N_n}=1$,
\begin{equation}\label{eq:block-L1}
  \|f_n'\|_{L^1(\mathbb{R})} = \varepsilon_n \ell_n \int_{-3/4}^{3/4} |\chi'(t) W_{N_n}(t) + 2\pi\chi(t) S_{N_n}(t)| \, dt
\leq \varepsilon_n \ell_n \bigl(\tfrac12\|\chi'\|_{\infty} + 4\pi \sqrt{N_n} \bigr)
\leq C_\chi\ell_n,
\end{equation}
where $C_\chi$ denotes a constant depending only on $\chi$, which may change from line to line.

Hence $\sum_{n}\|f_n'\|_{L^1(\mathbb{R})}\leq C_{\chi}\sum_{n}2^{-3n}<\infty$.
Together with the uniform convergence of $\sum_n f_n$, this shows that $f(x)=\int_{-\infty}^x\sum_n f_n'(t)\,dt$.
Thus $f$ is absolutely continuous on $\mathbb R$, and
$f'=\sum_{n}f_n'$ almost everywhere.

We next show that $f\in\lambda_*(\mathbb{R})$. By Claim~\ref{lem:WN-Zygmund} and Lemma~\ref{lem:Zygmund-multiplier},
\begin{equation}\label{eq:cutoff-uniform-Zygmund}
  \|\chi W_N\|_{\Lambda_*}\leq C_1
\end{equation}
with a constant $C_1$ independent of $N$.
Setting $s=h/\ell_n$, we have
\[
\Delta_h^2f_n(x)=\varepsilon_n\ell_n\,\Delta_s^2(\chi W_{N_n})(t),
\]
and therefore
\begin{equation}\label{eq:block-Zygmund}
\frac{|\Delta_h^2f_n(x)|}{|h|}
  =\varepsilon_n
    \frac{|\Delta_s^2(\chi W_{N_n})(t)|}{|s|}\leq C_1\varepsilon_n
\end{equation}
for all $x\in\mathbb{R}$ and $h\ne0$.
For fixed $x$ and $h$, the second difference $\Delta_h^2f_n(x)$ can be nonzero only if $J_n$ contains one of the points $x-h$, $x$, $x+h$.
Since the intervals $J_n$ are pairwise disjoint, at most three indices $n$ contribute to $\Delta_h^2f(x)=\sum_n\Delta_h^2f_n(x)$.
In particular, \eqref{eq:block-Zygmund} gives $\|f\|_{\Lambda_*}\le 3C_1/2$.

Let $\eta>0$, and choose $N_0$ such that $C_1/n<\eta$ for $n\geq N_0$.
For the finitely many indices $2 \leq n < N_0$, Taylor's theorem gives $|\Delta_h^2f_n(x)|\le\|f_n''\|_{\infty}h^2$, so
\[
\sum_{n=2}^{N_0-1}\frac{|\Delta_h^2f_n(x)|}{|h|}\le|h|\sum_{n=2}^{N_0-1}\|f_n''\|_{\infty}<\eta
\]
for all $x\in\mathbb R$ when $|h|$ is sufficiently small. The indices $n\ge N_0$ contribute at most $3\eta$ by \eqref{eq:block-Zygmund}.
Hence $|\Delta_h^2f(x)|/|h|<4\eta$ for all $x\in\mathbb R$ and all sufficiently small $|h|$, that is, $f\in\lambda_*(\mathbb R)$.

\subsection{Oscillation of the derivative}

We first show that $f' \notin \mathrm{VMO}(\mathbb{R})$. Let
\[
  I_n\coloneqq\bigl[ x_n-\tfrac12\ell_n,\, x_n+\tfrac12\ell_n\bigr].
\]
On $I_n$ we have $\chi=1$ and $\chi'=0$, and the other blocks vanish, so
$f'(x)=2\pi\varepsilon_n S_{N_n}\bigl((x-x_n)/\ell_n\bigr)$ for almost every $x\in I_n$.
Since each cosine in $S_N$ has integral zero over $[-1/2,1/2]$,
\begin{equation}\label{eq:endpoints}
f\bigl(x_n + \tfrac{1}{2}\ell_n\bigr) - f\bigl(x_n - \tfrac{1}{2}\ell_n\bigr) = \int_{I_n} f'(x) \, dx = 0.
\end{equation}
In particular $(f')_{I_n}=0$, and Claim~\ref{lem:lacunary-L1} gives
\begin{equation}\label{not_VMO}
\frac{1}{|I_n|}\int_{I_n}|f'(x)-(f')_{I_n}|\,dx= \frac{1}{|I_n|}\int_{I_n}|f'(x)|\,dx=2\pi\varepsilon_n\int_{-1/2}^{1/2}|S_{N_n}(t)|\,dt\geq\frac{2\pi}{\sqrt{6}}>1.
\end{equation}
Since $|I_n|\to0$, this shows $f'\notin\mathrm{VMO}(\mathbb R)$.

We now show that $f'\in\mathrm{BMO}(\mathbb R)$. Apply Lemma~\ref{lem:BMO-multiplier} with $\varphi=2\pi\chi$, $J=[-\frac34,\frac34]$, and $u=S_N$.
By Claim~\ref{lem:lacunary-BMO} and \eqref{eq:S-L1-J}, both $\|S_N\|_{\mathrm{BMO}(\mathbb R)}$ and $|J|^{-1}\int_J|S_N|$ are bounded by a constant multiple of $\sqrt N$.
Moreover, $\|\chi'W_N\|_{\mathrm{BMO}(\mathbb R)}\le2\|\chi'W_N\|_\infty\le\frac23\|\chi'\|_\infty$. Therefore
\begin{equation}\label{eq:derivative-fn}
   \left\|\chi'W_N+2\pi\chi S_N\right\|_{\mathrm{BMO}(\mathbb{R})}\leq C_{\chi}\sqrt N.
\end{equation}
Since the $\mathrm{BMO}$ seminorm is invariant under translations and dilations, \eqref{eq:fn-prime} and \eqref{eq:derivative-fn} imply
\begin{equation}\label{eq:uniform-block-BMO}
  \|f_n'\|_{\mathrm{BMO}(\mathbb{R})}\leq C_\chi\varepsilon_n\sqrt{N_n}=C_\chi.
\end{equation}

Let $I\subset\mathbb{R}$ be a bounded interval.
If $I$ meets at most one of the intervals $J_n$, then $f'$ coincides on $I$ with a single $f_n'$ (or vanishes), and \eqref{eq:uniform-block-BMO}
bounds the mean oscillation of $f'$ over $I$ by $C_\chi$.
Suppose that $I$ meets at least two of the intervals $J_n$.
Then every intersected interval $J_n$ is contained in $I$, except possibly the two intervals containing the endpoints of $I$.
If $J_n$ is such an exceptional interval, then $I$ contains the gap between $J_n$ and one of its neighbors, so \eqref{eq:gaps} gives
$|I|\ge\frac14\cdot2^{-n}\ge|J_n|$.
Using \eqref{eq:block-L1}, $\ell_n\le|J_n|$, and the disjointness of the intervals $J_n$, we find
\[
  \int_I|f'(x)|\,dx\leq \sum_{J_n\subset I}\|f_n'\|_{L^1(\mathbb R)}+\sum_{J_n\ \text{exceptional}}\|f_n'\|_{L^1(\mathbb R)}
  \leq C_\chi \Bigl(\sum_{J_n\subset I}|J_n|+ 2|I|\Bigr)\leq  3C_\chi|I|.
\]
Consequently,
\[
  \frac1{|I|}\int_I|f'-(f')_I|\,dx\leq\frac2{|I|}\int_I|f'|\,dx\leq 6C_\chi.
\]
Taking the supremum over all intervals $I$ gives $f'\in\mathrm{BMO}(\mathbb{R})$. This completes the proof of Theorem~\ref{thm:main}.

The construction in fact yields the following geometric statement.

\begin{theorem}\label{thm:geometric}
The graph of the function $f$ of Theorem~\ref{thm:main}
is an asymptotically conformal chord-arc curve that is not asymptotically smooth.
\end{theorem}

\begin{proof}
Since $f\in\lambda_*(\mathbb R)$, the graph of $f$ is an asymptotically conformal quasicircle by Theorem~\ref{graph-AS}.
Since $f'\in\mathrm{BMO}(\mathbb R)$, it is a chord-arc curve by \cite[Lemma 5.6]{MT2}.
Let $a_n$ and $b_n$ be the points of the graph over the endpoints of $I_n$. By \eqref{eq:endpoints}, $|a_n-b_n|=|I_n|$.
Since $s\mapsto\sqrt{1+s^2}$ is convex, Jensen's inequality and \eqref{not_VMO} give
\[
\frac{\ell(\wideparen{a_nb_n})}{|a_n-b_n|}=\frac{1}{|I_n|}\int_{I_n}\sqrt{1+f'(x)^2}\,dx
\geq\sqrt{1+\Bigl(\frac{1}{|I_n|}\int_{I_n}|f'(x)|\,dx\Bigr)^2}\geq\sqrt{1+\frac{2\pi^2}{3}}.
\]
As $|a_n-b_n|\to 0$ while these ratios stay above $\sqrt{1+2\pi^2/3}>1$, the graph of $f$ is not asymptotically smooth.
\end{proof}

\begin{remark}
The parameters $\varepsilon_n=1/n$ and $N_n=n^2$ are chosen so that $\varepsilon_n\to0$ and $\varepsilon_n\sqrt{N_n}=1$.
The first condition makes the contribution of the $n$-th block to the Zygmund seminorm tend to zero, which gives the little Zygmund condition.
On the other hand, the derivative of the $n$-th block is essentially $2\pi\varepsilon_n S_{N_n}$ rescaled, and by Claims~\ref{lem:lacunary-L1} and \ref{lem:lacunary-BMO} the mean oscillation of $S_{N_n}$ is comparable to $\sqrt{N_n}$.
The second condition therefore keeps the mean oscillation of $f'$ on the shrinking intervals $I_n$ bounded from above and from below.
Any parameters with $\varepsilon_n\to0$ and $\varepsilon_n\sqrt{N_n}\asymp1$ would serve the same purpose, since the argument based on Jensen's inequality in the proof of Theorem~\ref{thm:geometric} only uses a positive lower bound for the mean of $|f'|$ on $I_n$.
\end{remark}

\end{document}